\documentclass{amsart}

\usepackage{amsbsy,amsmath, amsfonts, latexsym, amsthm, amsxtra, amssymb}
\usepackage[numbers]{natbib}

\usepackage[shortlabels]{enumitem}
\SetEnumerateShortLabel{a}{\textup{(\alph*)}}
\SetEnumerateShortLabel{A}{\textup{(\Alph*)}}
\SetEnumerateShortLabel{1}{\textup{(\arabic*)}}
\SetEnumerateShortLabel{i}{\textup{(\roman*)}}
\SetEnumerateShortLabel{I}{\textup{(\Roman*)}}

\newtheorem{thm}{Theorem}[section]

\newtheorem{lem}[thm]{Lemma}
\newtheorem{prop}[thm]{Proposition}

\newtheorem{cor}[thm]{Corollary}
\newtheorem{conj}[thm]{Conjecture}

\theoremstyle{definition}

\newtheorem{df}[thm]{Definition}

\newtheorem{ex}[thm]{Example}

\theoremstyle{remark}
\newtheorem{rem}[thm]{Remark}
\newtheorem*{acknowledgement*}{Acknowledgement}

\def\calM{{\mathcal M}}

\def\frakm{{\mathfrak m}}

\def\frakn{{\mathfrak n}}

\def\NN{{\mathbb N}}

\def\QQ{{\mathbb Q}}

\def\ZZ{{\mathbb Z}}

\def\rmr{\mathrm{r}}

\def\rms{\mathrm{s}}

\def\Ap{\operatorname{Ap}}

\def\gr{\operatorname{gr}}

\def\H{\operatorname{H}} 

\def\Ht{\operatorname{ht}}

\def\lcm{\operatorname{lcm}}
\def\length{\operatorname{length}}

\def\LM{\operatorname{LM}} 

\def\minord{\operatorname{min-ord}}

\def\Quot{\operatorname{Quot}}

\def\spoly{\operatorname{spoly}}

\DeclareMathOperator{\ord}{ord}

\def\divides{\operatorname{|}}

\def\Index#1{\index{#1}\emph{#1}}
\def\Label#1{\label{#1}}
\def\Set#1{\left\{ #1\right \}}

\def\ideal#1{\langle #1 \rangle}

\def\isom{\cong}

\def\onto{\twoheadrightarrow}

\def\bd1{\mathbf{1}}

\def\LC{\operatorname{LC}} 
\def\LM{\operatorname{LM}} 
\def\LT{\operatorname{LT}} 

\def\Apery{Ap\'ery}

\title[Tangent Cone of Numerical Semigroup Rings]{Tangent Cone of Numerical Semigroup Rings of Embedding Dimension Three}
\author[Yi-Huang Shen]{Yi-Huang Shen}
\address{Department of Mathematics, University of Science and Technology of China, Hefei, Anhui, 230026, China} \email{yhshen@ustc.edu.cn}



\begin{document}

\begin{abstract}
   In this paper, we give new characterizations of the Buchsbaum and Cohen-Macaulay properties of the tangent cone $\gr_\frakm(R)$, where $(R,\frakm)$ is a numerical semigroup ring of embedding dimension $3$. In particular, we confirm the conjectures raised by Sapko on the Buchsbaumness of $\gr_\frakm(R)$.
\end{abstract}

\maketitle

\section{Introduction}
Throughout this paper let $\NN$ denote the set $\Set{0,1,2,\cdots}$. A \Index{numerical semigroup} $G$ generated by $n_1,\dots,n_d \in \NN$ is the set $\Set{\sum_{i=1}^n a_i n_i\mid  a_i \in \NN}$. It is a subsemigroup of $\NN$. For simplicity, we always assume that $G$ is minimally generated by these generators with $ n_1 < \dots < n_d$ and $\gcd(n_1, \ldots, n_d) = 1$, unless stated otherwise. Let $K$ be a field and $t$ an indeterminate over $K$. As a subring of the power series ring $V= K[ [t]]$, the ring $R=K[ [t^{n_1},\ldots,t^{n_d}]]$ is the \Index{numerical semigroup ring} associated to $G$  with $\frakm = (t^{n_1}, \ldots, t^{n_d})R$ being the unique maximal ideal. In this case, $d$ is the embedding dimension of $R$ and $n_1$ is the Hilbert-Samuel multiplicity of $R$ with respect to $\frakm$.

The numerical semigroup ring $R$ is a natural homomorphic image of the power series ring $S=K[ [x_1,\dots,x_d]]$. The kernel $I$ of this surjection is a binomial ideal (cf.~\citet[Corollary 7.3]{MR741678}) and it will be referred to as the \Index{defining ideal} of $R$.

Let $C$ be the monomial curve having the parametrization
\[
x_1=t^{n_1}, x_2=t^{n_2},\dots,x_d=t^{n_d}.
\]
To study the tangent cone of $C$ at the origin, it is very natural to take a closer look at the initial form ideal $I^*$ of the defining ideal $I$, which is the kernel of the natural homomorphism between the associated graded rings
\[
\gr_\frakn(S)=\bigoplus_{i=0}^\infty \frakn^i/\frakn^{i+1} \onto \gr_\frakm(R)=\bigoplus_{i=0}^\infty \frakm^i/\frakm^{i+1}.
\]
Here $\frakn$ and $\frakm$ are the maximal ideals of $S$ and $R$ respectively. The initial form ideal $I^*$ can be computed from $I$, for instance, by using the method in \citet[Section 15.10.3]{MR1322960}. We will also refer to $\gr_\frakm(R)$ as the \Index{tangent cone} of $R$.

Every two-generated numerical semigroup has a principal defining ideal. Hence, the first non-trivial example arises in the case when the embedding dimension $d=3$. In this situation, the defining ideal $I$ is always three-generated (cf.~ \citet{MR0269762}). Furthermore, \citet{MR632652} and \citet{MR578272} independently proved that $\gr_\frakm(R)$ is Cohen-Macaulay if and only if the minimal number of generators $\mu(I^*)\le 3$. We are able to provide an additional equivalent characterization in terms of reduction number and index of nilpotency. Recall that the ideal $Q=(t^{n_1})R$ is a principal reduction of the maximal ideal $\frakm$, with reduction number $r_Q(\frakm)=\min\Set{r\mid Q\frakm^r=\frakm^{r+1}}$ and index of nilpotency $s_Q(\frakm)=\min\Set{s\mid \frakm^{s+1}\subseteq Q}$.
It follows easily from the definition that $\rmr_Q(\frakm)\ge \rms_Q(\frakm)$, with equality when $\gr_\frakm(R)$ is Cohen-Macaulay (cf.~\citet[Corollary 2.7]{MR514892}). On the other hand, it is not very difficult to see that $\rmr_Q(\frakm)=\rms_Q(\frakm)$ in general will not lead to the Cohen-Macaulayness of $\gr_\frakm(R)$.  However, three-generated numerical semigroups turn out to be very special. We will prove in Theorem \ref{CM-cond-2} that when $R$ is a numerical semigroup ring of embedding dimension $3$, $\gr_\frakm(R)$ is Cohen-Macaulay if and only if $\rmr_Q(\frakm)=\rms_Q(\frakm)$.

We also study the Buchsbaum and 2-Buchsbaum properties in terms of the 0-th local cohomology modules. Let $\calM=\bigoplus_{i=1}^\infty \frakn^i/\frakn^{i+1}$ be the homogeneous maximal ideal of $\gr_\frakn(S)$. Then $\gr_\frakm(R)$ is said to be \Index{$k$-Buchsbaum} if $\calM^k \cdot \H_\calM^0(\gr_\frakm(R)) =0$. Normally, $1$-Buchsbaum will simply be referred to as Buchsbaum. It is evident that if $\length(\H_\calM^0(\gr_\frakm(R))\le 1$, then $\gr_\frakm(R)$ is Buchsbaum, and if $\length(\H_\calM^0(\gr_\frakm(R))\le 2$, then $\gr_\frakm(R)$ is 2-Buchsbaum. Interestingly enough, the converses are also true in this case.  We prove in Theorem \ref{length1} and \ref{length2}, that when $R$ is a numerical semigroup ring of embedding dimension $3$, the associated graded ring $\gr_\frakm(R)$ is Buchsbaum if and only if $\length(\H_\calM^0(\gr_\frakm(R)))\le 1$, and  $\gr_\frakm(R)$ is 2-Buchsbaum if and only if $\length(\H_\calM^0(\gr_\frakm(R)))\le 2$, respectively.
In particular, the three conjectures concerning the Buchsbaumness of the tangent cone $\gr_\frakm(R)$, raised by \citet{MR1855122}, are confirmed.

In \citet{MR2256395} and \citet{DANNA2009}, the authors introduced several invariants for the numerical semigroup $G$. Using these invariants, they gave various sufficient and/or necessary conditions for the tangent cone $\gr_\frakm(R)$ to be Cohen-Macaulay or Buchsbaum. As an application of our treatment, we will show in Theorem \ref{Buch-Another-Char}, that the sufficient condition in \citet[Theorem 3.8]{DANNA2009} is also necessary for the tangent cone $\gr_\frakm(R)$ to be Buchsbaum, under the further assumption that the embedding dimension $d=3$.

The main technique in this paper is to manipulate the standard basis of the defining ideal $I$. The standard basis algorithm can generate a standard basis from a binomial minimal generating set of $I$.  When the embedding dimension $d$ is small, it is possible to carry out the standard basis algorithm by hand. This approach turns out to be very useful for the investigation of the tangent cone when the embedding dimension is three. We will go over briefly related theory in the next section.

\section{Preliminaries}
\Label{IFI-Section}

Let us begin by explaining several key ingredients of the numerical semigroup $G=\ideal{n_1,\dots,n_d}$.  Recall that we assume $\gcd(n_1,\dots,n_d)=1$. Hence for every integer $g\gg 0$, we have $g\in G$. The integer $f=\max\Set{z\in\ZZ \mid z\not\in G}$ is called the \Index{Frobenius number} of $G$.  Let $e$ be a nonzero element in $G$. The \Index{{\Apery} set} of $G$ with respect to $e$ is $\Ap(G,e)=\Set{w_0,\cdots,w_{e-1}}$, where $w_i$ is the smallest element in $G$ congruent to $i$ modulo $e$. Sometimes we write the elements of $\Ap(G,e)$ in  increasing order: $\widetilde{w}_0=0 < \widetilde{w}_1< \cdots < \widetilde{w}_{e-1}=e+f$. The following lemma gives an important characterization of Gorenstein numerical semigroup rings.

\begin{lem}
    [{\cite{MR0017942,MR0414559,MR0265353}}]
    \Label{Symm}
    Let $G=\ideal{n_1,\dots,n_d}$ be a numerical semigroup and $R$ be its associated numerical semigroup ring. Furthermore, let $e=n_1$ be the multiplicity of $R$ and $f$ be the Frobenius number of $G$. Then the following conditions are equivalent.
\begin{enumerate}[a]
    \item The numerical semigroup ring $R$ is Gorenstein.
    \item The numerical semigroup $G$ is symmetric in the sense that for every $z\in \ZZ$, $z\in G$ if and only if $f-z \not\in G$.
    \item $\widetilde{w}_i+\widetilde{w}_{e-1-i}=\widetilde{w}_{e-1}$ for every integer $i$ such that $0\le i \le e-1$.
\end{enumerate}
\end{lem}

Recall that $\NN=\Set{0,1,2,\dots}$. For every $z\in G$, we have $z=\sum_i a_i n_i$ for some $a_i\in\NN$. We will frequently refer to such a linear combination as a  \Index{representation} of $z$ with respect to $G$. The integer $\sum_i a_i$ is called the \Index{length} of this representation. It is obvious that
$
\ord_\frakm(t^z)=\max\Set{\sum a_i \mid \sum a_i n_i =z, a_i \in \NN}.
$
When there is no confusion, we also write this number as $\ord_G(z)$ and similarly define $\minord_G(z)$ to be $\min\Set{\sum a_i \mid \sum a_i n_i =z, a_i \in \NN}$.
The ratio $\frac{\ord_G(z)}{\minord_G(z)}$ is called the \emph{elasticity of $z$ with respect to $G$}. We say $z=\sum_i a_i n_i$ is a \Index{maximal representation} of $z$ with respect to $G$ if $\sum a_i =\ord_G(z)$.

The semigroup $G$ can be equipped with a partial order $\leqq_G$: for $a,b\in G$, we write $a \leqq_G b$ if $b-a\in G$. This order relation was considered, for instance, in \citet{MR2105325}. Another important partial order is
$\preceq_G$: for $a,b\in G$, we write $a\preceq_G b$ if there exists an element $c$ in $G$ such that $a+c=b$ and $\ord_G(a)+\ord_G(c)=\ord_G(b)$.
The partial order $\preceq_G$ in this formulation was suggested by Lance Bryant.

\begin{lem}[{\citet[Corollary 3.20]{LanceBryant}}]
    \Label{Lance-thm}
    Let $(R,\frakm)$ be a Gorenstein numerical semigroup ring associated to a semigroup $G=\ideal{n_1,\dots,n_d}$, and assume that the associated graded ring $\gr_\frakm(R)$ is Cohen-Macaulay. Then $\gr_\frakm(R)$ is Gorenstein if and only if the following condition holds for $e=n_1$:
   \begin{align}
    \tag{\dag} \label{dag}
    w_i \preceq_G w_{e-1}\text{ for all } w_i\in \Ap(G,n_1).
  \end{align}
\end{lem}

\begin{rem} \Label{ela1}
  In the previous lemma, if the numerical semigroup $G$ is symmetric and the elasticity of $w_{e-1}$ with respect to $G$ is 1, then every representation of $w_{e-1}$ is maximal. Hence the condition \eqref{dag} holds automatically.
\end{rem}

The following remark to Lemma \ref{Lance-thm} will be useful for the proof of Theorem \ref{CM-cond-2}.

\begin{rem}
    \label{gor-rem}
    Let $G=\ideal{n_1,\dots,n_d}$ be a symmetric numerical semigroup  and $R$ its associated numerical semigroup ring. Suppose $\frakm$ is the maximal ideal of $R$ with a principal reduction $Q=(t^{n_1})R$.  If the condition \eqref{dag} holds and $\rms_Q(\frakm)=\rmr_Q(\frakm)$, then $\gr_\frakm(R)$ is Gorenstein. We do not need to assume that $\gr_\frakm(R)$ is Cohen-Macaulay in advance. For the proof, see \citet[Theorem 3.14]{LanceBryant}.
\end{rem}

For a Gorenstein numerical semigroup ring, the index of nilpotency $\rms_Q(\frakm)$ can be computed by using the $\frakm$-adic order of $w_{e-1}$.

\begin{lem}
    \Label{lm03}
    Let $(R,\frakm)$ be a Gorenstein numerical semigroup ring associated to the semigroup $G=\ideal{n_1,\dots,n_d}$, and let $f$ denote the Frobenius number of $G$. Then for the principal reduction $Q=(t^{n_1})R$ of $\frakm$, $\rms_Q(\frakm)=\ord_G({f+n_1})$.
\end{lem}

\begin{proof}
    We always have \[
    \rms_Q(\frakm)=\max\Set{\ord_G(w)\mid 0\ne w\in \Ap(G,n_1)}.
    \]
    When $G$ is symmetric, the maximum is obviously achieved at $\ord_G({f+n_1})$.
\end{proof}

By convention, let $S=K[ [x_1,\dots,x_d]]$ be a power series ring over a field $K$ and $\frakn$ be its maximal ideal. This ring maps naturally onto the numerical semigroup ring $R=K[ [t^{n_1},\dots,t^{n_d}]]$. For each nonzero  element $x\in S$, let $o=\ord_\frakn(x)<\infty$ be the $\frakn$-adic order of $x$. We denote by $x^*$ the residue class of $x$ in $\frakn^o/\frakn^{o+1}$ and call it the \Index{initial form} of $x$. The \Index{initial form ideal} $I^*\subseteq \gr_\frakn(S)$ is generated by $x^*$ for all $x\in I$, and $\gr_\frakm(R)\isom \gr_\frakn(S)/I^*$ canonically. For our numerical semigroup ring $R$, the radical of the initial ideal $I^*$ is very simple.

\begin{lem}
    \Label{Radical}
    $\sqrt{I^*}=\ideal{x_2^*,\ldots,x_d^*}\gr_\frakn(S) $.
\end{lem}

\begin{proof}
    Consider the binomials $f_i:=x_i^{n_1}-x_1^{n_i}\in I$ for $2\le i \le d$. Since $n_1<n_i$, the initial form of $f_i^*$ is $(x_i^{*})^{n_1}\in I^*$. Hence $\ideal{x_2^*,\ldots,x_d^*}\subseteq \sqrt{I^*}$. Since $\Ht (\ideal{x_2^*,\dots,x_d^*}\gr_\frakn(S))=\Ht (I^*) =\Ht(I)=d-1$ and $\ideal{x_2^*,\ldots,x_d^*}$ is a prime ideal in $\gr_\frakn(S)$, $\sqrt{I^*}=\ideal{x_2^*,\ldots,x_d^*} \gr_\frakn(S)$.
\end{proof}

Since $\gr_\frakn(S)\isom K[x_1,\dots,x_d]$, for ease of notation,  we will simply write $x_i^*$ as $x_i$ for $1\le i \le d$ when there is no confusion.

Now we recall briefly some of the concepts and facts related to monomial orders and standard bases. Readers who are unfamiliar with these topics may wish to consult \citet{MR2363237}.

\begin{df}
Let $T=K[\underline{X}]=K[x_1,\dots,x_d]$ be a polynomial ring over a field $K$.
\begin{enumerate}[a]
\item A total order $>_\tau$ on the set of monomials $\Set{\underline{X}^\alpha\mid \alpha \in \NN^d}\subseteq T$ is a \Index{monomial order} if
$\underline{X}^\alpha >_\tau \underline{X}^\beta \implies \underline{X}^{\alpha+\gamma} >_\tau \underline{X}^{\beta+\gamma}$ for any $\alpha, \beta ,\gamma \in \NN^d$.

\item A monomial order $>_\tau$ is a \Index{local order} if $1>_\tau \underline{X}^\alpha$ for all $\alpha\ne \underline{0}\in \NN^d$; it is a \Index{global order} or \Index{term order} if $1<_\tau \underline{X}^\alpha$ for all $\alpha\ne \underline{0}\in \NN^d$.

    \item A local monomial order $>_\tau$ is \Index{degree compatible} if $\deg({{\underline{X}}}^\alpha)< \deg({{\underline{X}}}^\beta) \implies {{\underline{X}}}^\alpha >_\tau {{\underline{X}}}^\beta$ for any $\alpha, \beta \in \NN^d$.
\end{enumerate}
\end{df}

The negative degree reverse lexicographic order that we shall introduce here is a very useful monomial order. It is local and degree compatible.

\begin{df}
A monomial order $>_{ds}$ such that
\begin{align*}
    {{\underline{X}}}^\alpha >_{ds} {{\underline{X}}}^\beta \Longleftrightarrow
    \Big\{\deg({{\underline{X}}}^\alpha) < \deg ({{\underline{X}}}^\beta)
    \text{ or }\Big(\deg ({{\underline{X}}}^\alpha) = \deg({{\underline{X}}}^\beta) \text{ and }
    \exists 1\le i \le d: \\
    \alpha(n)=\beta(n),\dots,\alpha(i+1)=\beta(i+1),\alpha(i)<\beta(i)\Big)
    \Big\},
\end{align*}
is called a \Index{negative degree reverse lexicographic order} on $T=K[x_1,\dots,x_d]$.
\end{df}

Fix a monomial order $>_\tau$ on $T$. For a nonzero polynomial $f=\sum c_\alpha \underline{X}^\alpha$, the \Index{leading monomial} of $f$ is $\LM(f):=\max_{>_\tau}\Set{\underline{X}^\alpha\mid c_\alpha\ne 0}$. When $\LM(f)=\underline{X}^\alpha$, we call $\LC(f):=c_\alpha$ the \Index{leading coefficient} of $f$ and $\LT(f):=c_\alpha \underline{X}^\alpha$ the \Index{leading term} of $f$.
For two nonzero polynomials $f$ and $g$ in $T$, the $s$-polynomial is defined as follows:
\[
\spoly(f,g):=\frac{\lcm(\LM(f),\LM(g))}{\LT(f)} f - \frac{\lcm(\LM(f),\LM(g))}{\LT(g)}g.
\]

A finite set $B\subseteq I$ is a \Index{standard basis} of an ideal $I\subseteq T$ if for any nonzero $f\in I$, there exists an element $g\in B$ satisfying $\LM(g) \divides \LM(f)$. The famous Buchberger's criterion (cf.~ \citet[Theorem 1.7.3]{MR2363237}) says that a generating set $B=\Set{g_1,\dots,g_t}$ of $I$ is a standard basis if and only there exist $c_{ijk}\in T$ such that for all $i$ and $j$, $s(g_i,g_j)=\sum_{k} c_{ijk} g_k$ and $\LM(c_{ijk}g_k) <_\tau \LM(s(g_i,g_j))$ when $c_{ijk}\ne 0$. With a global monomial order, a standard basis $B$ can always be generated from a generating set $B_0=\Set{g_1,\dots,g_{t'}}$ of $I$ by applying the standard basis algorithm (cf.~\citealt[Section 1.7]{MR2363237}). Roughly speaking, this algorithm extends the generating set $B_0$ to the standard basis $B$ by successively adding nonzero $s$-polynomials $\spoly(g_i,g_j)$. Since new generators will also be needed for calculating the $s$-polynomials,
when working with a local monomial order, the standard basis algorithm might not terminate in finite steps. Nevertheless, this is not a problem when dealing with the defining ideal of a numerical semigroup ring. The finiteness is guaranteed by the algorithm described in Section 15.10.3 of \citet{MR1322960} with a global monomial order. The algorithm in \citet{MR1322960} uses  a homogenization technique and has been implemented in the package {\tt TangentCone} of {\tt Macaulay2} \cite{M2}. Furthermore, one can remove the redundancies in the standard basis obtained here and arrives at a reduced standard basis which is uniquely determined, see Definition 1.6.2 and Exercise 1.6.1 of \citet{MR2363237} for clarity.

From now on, fix a numerical semigroup $G=\ideal{n_1,\dots,n_d}$.

\begin{df}
  For distinct $\alpha=(\alpha(1),\dots,\alpha(d))$ and $\beta=(\beta(1),\dots,\beta(d))\in \NN^d$, the binomial $f={{\underline{X}}}^{\alpha}-{{\underline{X}}}^\beta \in T$ is \Index{weakly balanced} with respect to $G$  if $\sum_i \alpha(i) n_i = \sum_i \beta(i) n_i$. The binomial $f$ is called \Index{balanced} if it is weakly balanced, $\deg({{\underline{X}}}^\alpha)=\deg({{\underline{X}}}^\beta)$, and ${{\underline{X}}}^\alpha$ and ${{\underline{X}}}^\beta$ are coprime.
\end{df}

For the numerical semigroup ring $R$, the defining ideal $I$ is generated by weakly balanced binomials (cf.~\citet[Corollary 7.3]{MR741678}). If we fix a degree compatible local monomial order  and apply the standard basis algorithm (cf.~\citet[Section 1.7]{MR2363237}) to this generating set, we are able to obtain a reduced standard basis $\Set{f_1,\dots,f_s}$. In this case, the initial form ideal $I^*$ is minimally generated by the corresponding initial forms:
\[
I^*=\ideal{f_1^*,\dots,f_s^*}\gr_\frakn(S).
\]
Since each $f_i$ is also a weakly balanced binomial, $f_i^*$ is either a monomial or a balanced binomial. In the latter case, roughly speaking, $f_i=f_i^*$. In the rest of this paper, when we say that $g$ is a \Index{minimal generator} of $I^*$, it is understood that $g\in \Set{f_1^* ,\dots,f_s^*}$ when the minimal binomial generating set of $I$ and the monomial order is clear.

Our next task is to choose a suitable monomial order $>_\tau$ for $K[x_1,\dots,x_d]$.

\begin{df}
  A local monomial order $>_\tau$ is \Index{nice} in the variable $x_i$ if the following holds:
\[
\left\{
\deg ({{\underline{X}}}^\alpha) < \deg ({{\underline{X}}}^\beta) \text{ or } \left({{\underline{X}}}^\alpha-{{\underline{X}}}^\beta \text{ is balanced, } \beta(i)>0 \right)
\right\}
\implies {{\underline{X}}}^\alpha >_\tau {{\underline{X}}}^\beta.
\]
\end{df}

Being nice is really a mild condition. For instance, the following monomial order is nice in $x_1$:
\begin{align*}
    {{\underline{X}}}^\alpha > {{\underline{X}}}^\beta \stackrel{def}{\Longleftrightarrow}
    \Big\{\deg({{\underline{X}}}^\alpha) < \deg ({{\underline{X}}}^\beta)
    \text{ or }\Big(\deg ({{\underline{X}}}^\alpha) = \deg({{\underline{X}}}^\beta) \text{ and }
    \exists 1\le i \le d: \\
    \alpha(1)=\beta(1),\dots,\alpha(i-1)=\beta(i-1),\alpha(i)<\beta(i)\Big)
    \Big\}.
\end{align*}
When $d=3$, the negative degree reverse lexicographic order is also nice in $x_1$.

\begin{df}
  Let $K[x_1,\dots,x_d]$ be a polynomial ring with a degree compatible local monomial order $>_\tau$ and $G=\ideal{n_1,\dots,n_d}$ the underlying numerical semigroup. An ideal $J\subseteq K[x_1,\dots,x_d]$ is called \Index{almost balanced} if it satisfies the following two conditions:
\begin{enumerate}[a]
    \item $\sqrt{J}=(x_2,\dots,x_d)$;
    \item there is a reduced standard basis $\Set{f_1,\dots,f_t}$ of $J$ such that  $f_i$ is either a monomial or a balanced binomial.
\end{enumerate}
\end{df}

For a numerical semigroup ring $(R,\frakm)$ and its defining ideal $I$, it is clear that the initial form ideal $I^*$ is almost balanced. The following lemma is crucial when discussing the Cohen-Macaulayness of $\gr_\frakm(R)$.

\begin{lem}
    \Label{CM-cond-1}
    Let $>_\tau$ be a local monomial order for $T=K[x_1,\dots,x_d]$ that is nice in $x_1$,  and $J$ an almost balanced $T$-ideal. Suppose $\Set{f_1,\dots,f_t}$ forms a reduced standard basis of $I$ as in the previous definition. Then $T/J$ is Cohen-Macaulay if and only if for any $f_i$, either $f_i$ is binomial or $x_1$ does not divide $f_i$.
\end{lem}

\begin{proof}
  Observe that $J$ is homogeneous. Hence for the homogeneous maximal ideal $\frakm=\ideal{x_1,\dots,x_d}$ of $T$,  $T/J$ is Cohen-Macaulay if and only if $(T/J)_\frakm$ is Cohen-Macaulay, if and only if $x_1$ is a regular element on $(T/J)_\frakm$.

   If some $f_i=x_1 {{\underline{X}}}^{\alpha}$, then ${{\underline{X}}}^\alpha \not\in J$ since $\Set{f_1,\dots,f_s}$ is a reduced standard basis. Therefore $J$ is not a perfect ideal.

    Conversely, suppose that $x_1 f\in J_\frakm$ and $0\ne f\not\in J_\frakm$. By multiplying suitable unit element in $T_\frakm$, we may assume that $f\in T$.  Notice that $x_1\LM({f})=\LM(x_1{f})$ is divisible by some $\LM(f_i)$. But $\LM({f})$ is not divisible by this $\LM(f_i)$, hence $\LM(f_i)$ is divisible by $x_1$. Since the monomial order $>_\tau$ is nice in $x_1$, $f_i$ cannot be a balanced binomial. Hence it is a monomial.
\end{proof}

\begin{ex}\Label{exam243}
    Let $K$ be a field, $R=K[ [t^5,t^6,t^{13}]]$ and $\frakm=(t^5,t^6,t^{13})R$. Then the defining ideal is
    \[
    I=(x_2^2x_3-x_1^5,x_3^2-x_1^4x_2,x_1x_3-x_2^3) \subseteq K[ [ x_1,x_2,x_3]].
    \]
    With respect to the negative degree reverse lexicographic order, the set
    \[
    \Set{x_2^2x_3-x_1^5,x_3^2-x_1^4x_2,x_1x_3-x_2^3,x_2^5-x_1^6}
    \]
    forms a reduced standard basis of $I$. Thus the initial form ideal is
    \[
    I^*=(x_2^2x_3,x_3^2,x_1x_3,x_2^5,)\subseteq K[x_1,x_2,x_3].
    \]
    Since the generator $x_1x_3$ is divisible by $x_1$, $\gr_\frakm(R)$ is not Cohen-Macaulay. This non-Cohen-Macaulay property also follows immediately from the fact that $I^*$ is generated by more than 3 elements.
\end{ex}

The $\alpha$-invariants of the numerical semigroup $G$ will also be needed in our investigation.

\begin{df}
    \Label{alpha-inv}
    For the numerical semigroup $G=\ideal{n_1,\cdots ,n_d}$, define
 \[
 \alpha_i = \min\Set{\alpha\in \NN \mid \alpha n_i \in \ideal{n_1,\dots,\widehat{n_i}, \dots,n_d}, \alpha\ne 0}
\]
for $1\le i \le d$.  Here for the ``truncated'' semigroup $\ideal{n_1,\dots,\widehat{n_i}, \dots,n_d}$, we do not require that $\gcd\Set{n_j\mid j\ne i}=1$.
\end{df}

In Example \ref{exam243}, we have $\alpha_1=5$, $\alpha_2=3$ and $\alpha_3=2$.

\section{When the Embedding Dimension $d=3$}

In this section, we will always use the negative degree reverse lexicographic order on $\gr_\frakn(S)\isom K[x_1,x_2,x_3]$. Hence if $f=x_2^b -x_1^a x_3^c$ is a balanced binomial with respect to the numerical semigroup $G=\ideal{n_1,n_2,n_3}$, then the leading monomial of $f$ is $x_2^b$.

The basis of the initial form ideal $I^*$ will be constructed as in Section \ref{IFI-Section} from the binomial basis given in the following fundamental theorem for three-generated numerical semigroups.

\subsection{Fundamental Theorem}

\begin{thm}
    [{\cite{MR0269762}}]
    \Label{herzog-thm}
    Let $R$  be a numerical semigroup ring corresponding to the numerical semigroup $G=\ideal{n_1,n_2,n_3}$. Then for the $\alpha$-invariants $\alpha_i$ as defined in Definition \ref{alpha-inv}, and suitable numbers $\alpha_{ij}\in \NN$ where $1\le i,j\le 3$, the following conditions hold.
    \begin{enumerate}[a]
        \item If $R$ is Gorenstein, then, after a permutation $(i,j,k)$ of $(1,2,3)$, the defining ideal is
            \[
            I=(x_i^{\alpha_i}-x_j^{\alpha_j}, x_k^{\alpha_k}-x_i^{\alpha_{ki}}x_j^{\alpha_{kj}}),
            \]
            and the Frobenius number of $G$ is
            \[
            f=(\alpha_i-1) n_i + (\alpha_k-1) n_k -n_j.
            \]
        \item If $R$ is not Gorenstein, then
            \[
            I=(x_1^{\alpha_1}-x_2^{\alpha_{12}}x_3^{\alpha_{13}}, x_2^{\alpha_2}-x_1^{\alpha_{21}}x_3^{\alpha_{23}}, x_3^{\alpha_3}-x_1^{\alpha_{31}}x_2^{\alpha_{32}}),
            \]
            where $\alpha_i=\alpha_{ji}+\alpha_{ki}$ for all permutation $(i,j,k)$ of $(1,2,3)$. Furthermore, each $\alpha_{ij}>0$ for $1\le i\ne j \le 3$.
    \end{enumerate}
\end{thm}

The invariants appeared in this theorem have been studied extensively, for instance, in \citet{MR2385494} and \citet{MR2105325}. Now, applying Lemma \ref{CM-cond-1} to Theorem \ref{herzog-thm}, one can quickly give arithmetic conditions for $\gr_\frakm(R)$ to be Cohen-Macaulay.

\begin{cor} Resume the notation from Theorem \ref{herzog-thm}.
    \Label{d3cm}
    \begin{enumerate}[a]
        \item If $I=(x_1^{\alpha_1}-x_2^{\alpha_2}, x_3^{\alpha_3}-x_1^{\alpha_{31}}x_2^{\alpha_{32}})$, then $\gr_\frakm(R)$ is a complete intersection and $I^*$ is generated by $\Set{x_2^{\alpha_2},x_3^{\alpha_3}}$.
        \item If $I=(x_1^{\alpha_1}-x_3^{\alpha_3}, x_2^{\alpha_2}-x_1^{\alpha_{21}}x_3^{\alpha_{23}})$, then $\gr_\frakm(R)$ is Cohen-Macaulay if and only if $\alpha_2 \le \alpha_{21}+\alpha_{23}$. When $\gr_\frakm(R)$ is Cohen-Macaulay, $I^*$ is generated by $ \Set{x_3^{\alpha_3},(x_2^{\alpha_2}-x_1^{\alpha_{21}}x_3^{\alpha_{23}})^*}$.

        \item If $I=(f_1:=x_2^{\alpha_2}-x_3^{\alpha_3}, f_2:=x_1^{\alpha_1}-x_2^{\alpha_{12}}x_3^{\alpha_{13}})$, we can always assume that $\alpha_{13}< \alpha_3$. Then $\gr_\frakm(R)$ is Cohen-Macaulay if and only if $\alpha_2+\alpha_{12}\le \alpha_1 + \alpha_3 -\alpha_{13}$.  Set $f_3:=x_2^{\alpha_2+\alpha_{12}}-x_1^{\alpha_{1}}x_3^{\alpha_3-\alpha_{13}}$, the $s$-polynomial of $f_1$ and $f_2$. When $\gr_\frakm(R)$ is Cohen-Macaulay, $I^*$ is generated by $\Set{x_3^{\alpha_3},x_2^{\alpha_{12}}x_3^{\alpha_{13}},f_3^*}$.

        \item If $I=(x_1^{\alpha_1}-x_2^{\alpha_{12}}x_3^{\alpha_{13}}, x_2^{\alpha_2}-x_1^{\alpha_{21}}x_3^{\alpha_{23}}, x_3^{\alpha_3}-x_1^{\alpha_{31}}x_2^{\alpha_{32}})$, then $\gr_\frakm(R)$ is Cohen-Macaulay if and only if $\alpha_2 \le \alpha_{21}+\alpha_{23}$. When $\gr_\frakm(R)$ is Cohen-Macaulay, $I^*$ is generated by $\Set{x_2^{\alpha_{12}}x_3^{\alpha_{13}}, (x_2^{\alpha_2}-x_1^{\alpha_{21}}x_3^{\alpha_{23}})^*, x_3^{\alpha_3}}$.
    \end{enumerate}
\end{cor}

\subsection{Cohen-Macaulayness}
The purpose of this subsection is to establish a new characterization for the Cohen-Macaulayness of $\gr_\frakm(R)$ when the embedding dimension $d=3$. Notice that the ideal $Q=(t^{n_1})R$ is a principal reduction of the maximal ideal $\frakm$. We want to connect the Cohen-Macaulay property with the reduction number $\rmr_Q(\frakm)=\min\Set{r\mid Q\frakm^r=\frakm^{r+1}}$ and the index of nilpotency $\rms_Q(\frakm)=\min\Set{s\mid \frakm^{s+1}\subseteq Q}$. It follows easily from the definition that $\rmr_Q(\frakm)\ge \rms_Q(\frakm)$. When $\gr_\frakm(R)$ is Cohen-Macaulay, a result of \citet[Corollary 2.7]{MR514892} implies that $\rmr_Q(\frakm)=\rms_Q(\frakm)$. On the other hand, it is not very difficult to see that $\rmr_Q(\frakm)=\rms_Q(\frakm)$ in general will not lead to the Cohen-Macaulayness of $\gr_\frakm(R)$. However, it is different for a numerical semigroup ring of embedding dimension $3$.

\begin{thm}
    \Label{CM-cond-2}
    Suppose $(R,\frakm)$ is a numerical semigroup ring of embedding dimension 3, and $Q=(t^{n_1})R$ is a principal reduction of the maximal ideal $\frakm$. The tangent cone $\gr_\frakm(R)$ is Cohen-Macaulay if and only if the index of nilpotency $\rms_Q(\frakm)$ equals the reduction number $\rmr_Q(\frakm)$.
\end{thm}

\begin{proof}
  Let $G=\ideal{n_1,n_2,n_3}$ be the associated numerical semigroup.  Since the  ``only if'' part is clear, we may assume that $\rms_Q(\frakm)=\rmr_Q(\frakm)$ and proceed to show that $\gr_\frakm(R)$ is Cohen-Macaulay.  For every $x\in S=K[ [x_1,x_2,x_3]]$, we will write $\overline{x}$ for its image in $R=S/I$, where $I$ is the defining ideal. The integer $e$ will be the multiplicity $n_1$.

\begin{enumerate}[1]
    \item First, we study the case when $G$ is symmetric, i.e., the numerical semigroup ring $R$ is a complete intersection.  Now, for the {\Apery} set element $w_{e-1}=f+n_1$ where $f$ is the Frobenius number of $G$, $\ord_G(w_{e-1})=\rms_Q(\frakm)$ by Lemma \ref{lm03}.  Using the same notation as in Theorem \ref{herzog-thm}, we have three cases.

    \begin{enumerate}[i]
        \item When $(i,j,k)=(1,2,3)$, the tangent cone is automatically a complete intersection.
        \item Suppose $(i,j,k)=(1,3,2)$. Now the last {\Apery} set element can be written as
            \[
            w_{e-1}=(\alpha_2-1)n_2+(\alpha_3-1)n_3
            \]
            by part (a) of Theorem \ref{herzog-thm}.  This is obviously the unique representation of $w_{e-1}$ with respect to $G$. It follows from Remark \ref{ela1} that the condition \eqref{dag} holds. When $\rmr_Q(\frakm)=\rms_Q(\frakm)$, $\gr_{\frakm}(R)$ is Gorenstein by Remark \ref{gor-rem}.
          \item Suppose $(i,j,k)=(2,3,1)$. We may again assume that $\alpha_{13}<\alpha_3$. Now the Frobenius number can be written as
            \[
            f=(\alpha_2+\alpha_{12}-1)n_2 + (\alpha_{13}-1)n_3 -n_1,
            \]
            therefore
            \[
            w_{e-1}=(\alpha_2+\alpha_{12}-1)n_2 + (\alpha_{13}-1)n_3.
            \]
            This gives the maximal representation of $w_{e-1}$ with respect to $G$, and $\rms_Q(\frakm)=\ord_\frakm(w_{e-1}) =(\alpha_2+\alpha_{12}-1)+(\alpha_{13}-1)$ by Lemma \ref{lm03}.  In this case, the tangent cone $\gr_\frakm(R)$ is Cohen-Macaulay if and only if $\alpha_2 + \alpha_{12}\le \alpha_1 + \alpha_3-\alpha_{13}$ by Corollary \ref{d3cm}(c).  For this subcase, we want to prove that the following conditions are equivalent:
            \begin{enumerate}[a]
                \item $\gr_\frakm(R)$ is Cohen-Macaulay.
                \item $\rmr_Q(\frakm)=\rms_Q(\frakm)$.
                \item $\rmr_Q(\frakm) \le \alpha_1 + \alpha_3 -2$.
            \end{enumerate}

(a) $\Rightarrow$ (b): It is clear.

(b) $\Rightarrow$ (c): For $r=\rmr_Q(\frakm)$, $\overline{x_2^{\alpha_2+\alpha_{12}-1}x_3^{\alpha_{13}-1}}\in \frakm^r$. So $\overline{x_2^{\alpha_2+\alpha_{12}}x_3^{\alpha_{13}-1}}\in \frakm^{r+1}=Q\frakm^r$. But
            \[
            \overline{x_2^{\alpha_2+\alpha_{12}}x_3^{\alpha_{13}-1}} =\overline{x_2^{\alpha_{12}}x_3^{\alpha_3 + \alpha_{13}-1}}=\overline{x_1^{\alpha_1}x_3^{\alpha_3-1}},
            \]
            and $\overline{x_1}$ is a regular element in the domain $R$. Hence $\overline{x_1^{\alpha_1-1}x_3^{\alpha_3-1}}\in \frakm^r$. We want to show that $\ord_\frakm(\overline{x_1^{\alpha_1-1}x_3^{\alpha_3-1}})=(\alpha_1-1)+(\alpha_3-1)$, so that (c) holds.  It suffices to show that $(\alpha_1-1)n_1+(\alpha_3-1)n_3$ is the unique representation of this element with respect to $G$. Suppose not, then
            \[
            (\alpha_1-1)n_1+(\alpha_3-1)n_3=an_1+bn_2+cn_3,
            \]
            with $a,b,c\in \NN$ and $b>0$. By the minimality of $\alpha_1$ and $\alpha_3$, one must have $a\le \alpha_1-1$ and $c\le \alpha_3-1$. Now
            \[
            (\alpha_1-1-a)n_1+(\alpha_3-1-c)n_3 =bn_2
            \]
            Since $b>0$,  $b\ge \alpha_2$ by the minimality of $\alpha_2$. Hence
            \[
            (\alpha_1-1-a)n_1+(\alpha_3-1-c)n_3 = (b-\alpha_2)n_2+ \alpha_3 n_3,
            \]
            thus
            \[
            (\alpha_1-1-a)n_1 = (b-\alpha_2)n_2+ (c+1) n_3,
            \]
            which contradicts the minimality of $\alpha_1$. This shows that (b) implies (c).

        (c) $\Rightarrow$ (a): We have $\alpha_2+\alpha_{12}+\alpha_{13}-2 = \rms_Q(\frakm)\le \rmr_Q(\frakm)\le \alpha_1 + \alpha_3 -2$. It follows immediately that  $\alpha_2 + \alpha_{12}\le \alpha_1 + \alpha_3 -\alpha_{13}$. Hence $\gr_Q(\frakm)$ is Cohen-Macaulay and (a) holds.

    \end{enumerate}

\item Next, we consider the case when the semigroup group $G$ is not symmetric. Recall that the defining ideal is
    \[
    I=(f_1:=x_1^{\alpha_1}-x_2^{\alpha_{12}}x_3^{\alpha_{13}}, f_2:=x_2^{\alpha_2}-x_1^{\alpha_{21}}x_3^{\alpha_{23}}, f_3:=x_3^{\alpha_3}-x_1^{\alpha_{31}}x_2^{\alpha_{32}}).
    \]
    Our aim is to show that if $\rms_Q(\frakm)=\rmr_Q(\frakm)$, then $\alpha_2\le \alpha_{21}+\alpha_{23}$.
    First of all, with the partial order $\leqq_G$ we introduced in Section 2, we have
    \[
        \max_{\leqq_{G}} \Ap(G,n_1)= \Set{(\alpha_2-1)n_2 + (\alpha_{13}-1)n_3, (\alpha_3-1)n_3 + (\alpha_{12}-1)n_2},
    \]
    from \citet[Lemma 4]{MR2105325}.  Therefore, the index of nilpotency is
\[
      \rms_Q(\frakm)=\max\Set{\alpha_2+\alpha_{13}-2,\alpha_3+\alpha_{12}-2},
\]
    by a proof similar to that of Lemma \ref{lm03}. Now we are ready to complete the proof.
    \begin{enumerate}[i]
        \item The case when $\rms_Q(\frakm)=\alpha_2+\alpha_{13}-2$ is easy. Suppose the condition is satisfied, i.e., $\alpha_2+{\alpha_{13}}-2=r=\rmr_Q(\frakm)$. Then $\overline{x_2^{\alpha_{2}}x_3^{\alpha_{13}-1}}\in \frakm^{r+1}=Q\frakm^r$. Notice that $\overline{x_2^{\alpha_{2}}x_3^{\alpha_{13}-1}}=\overline{x_1^{\alpha_{21}}x_3^{\alpha_3-1}}$. Hence $\overline{x_1^{\alpha_{21}-1}x_3^{\alpha_3-1}}\in \frakm^r$. Similar to the  Gorenstein case, one can show that the representation $z=(\alpha_{21}-1)n_1+(\alpha_3-1)n_3\in G$ is unique, hence $\ord_\frakm(\overline{x_1^{\alpha_{21}-1}x_3^{\alpha_3-1}}) =\alpha_{21}+\alpha_3-2\ge r=s=\alpha_2+\alpha_{13}-2$. Thus $\alpha_2 \le \alpha_{21}+\alpha_3-\alpha_{13}=\alpha_{21}+\alpha_{23}$, and $\gr_\frakm(R)$ is Cohen-Macaulay.

        \item If $\rms_Q(\frakm)>\alpha_2+\alpha_{13}-2$ and $\rmr_Q(\frakm)=\rms_Q(\frakm)$, then let $\delta:=\alpha_{23}-\alpha_{32}>0$ and thus $\alpha_2+\alpha_{13}-2=r-\delta$. Now $\ord_\frakm(\overline{x_2^{\alpha_{2}-1}x_3^{\alpha_{13}-1}})\ge r-\delta$, hence $\overline{x_2^{\alpha_{2}+\delta}x_3^{\alpha_{13}-1}} \in\frakm^{r+1}=Q\frakm^r$. It follows easily that $\overline{x_1^{\alpha_{21}-1}x_2^\delta x_3^{\alpha_{3}-1}}\in \frakm^r$.  Suppose to the contrary that $\gr_\frakm(R)$ is not Cohen-Macaulay, then $\alpha_2=\alpha_{12}+\alpha_{32}>\alpha_{21}+\alpha_{23}$. Hence $\delta=\alpha_{23}-\alpha_{32}< \alpha_{12}-\alpha_{21}<\alpha_{12}$. We claim that the representation
            \[
            P: z=(\alpha_{21}-1)n_1+\delta n_2 +(\alpha_3-1)n_3\in G
            \]
            is maximal. Let
            \[
            Q: z=an_1+bn_2+cn_3
            \]
            be a distinct representation of $z$. Then we have $6$ cases when comparing the coefficients of $P$ and $Q$. The proof of the claim is straightforward and easy. To avoid unnecessary repetition, we just consider the exemplifying case where $\alpha_{21}-1\ge a$, $\delta < b$ and $\alpha_3-1\ge c$. Whence
            \[
            (b-\delta)n_2=(\alpha_{21}-1-a)n_1 + (\alpha_3-1-c)n_3.
            \]
            By the choice of $\alpha_2$, $b-\delta\ge \alpha_2$, hence
            \[
            (b-\delta-\alpha_2)n_2 = (\alpha_{21}-1-a-\alpha_{21})n_1 + (\alpha_3-1-c-\alpha_{23})n_{23}.
            \]
            Or equivalently
            \[
            (a+1)n_1+ (b-\delta-\alpha_2)n_2 = (\alpha_{13}-1-c)n_3.
            \]
            This implies that $0<\alpha_{13}-1-c<\alpha_3$, which is against the choice of $\alpha_3$. The argument for other cases is similar.
           Now $\ord_\frakm(\overline{x_1^{\alpha_{21}-1}x_2^\delta x_3^{\alpha_{3}-1}})=\delta+\alpha_{21}+\alpha_3-2\ge r=s=\alpha_2+\alpha_{13}-2+\delta$. Hence $\alpha_2 \le \alpha_{21}+\alpha_3-\alpha_{13}=\alpha_{21}+\alpha_{23}$, and $\gr_\frakm(R)$ is again Cohen-Macaulay.
    \end{enumerate}
\end{enumerate}
\end{proof}

We thank Lance Bryant for the helpful comments regarding Theorem \ref{CM-cond-2}.

\begin{ex}
    Let $K$ be a field, $R=K[ [t^7,t^{10},t^{25}]]$. Then $Q=(t^7)R$ is a principal reduction of the maximal ideal $\frakm=(t^7,t^{10},t^{25})R$. We have $\rmr_Q(\frakm)=\rms_Q(\frakm)=5$, hence $\gr_\frakm(R)$ is Cohen-Macaulay by Theorem \ref{CM-cond-2}.
\end{ex}

The statement of Theorem \ref{CM-cond-2} fails if the embedding dimension is 4.

\begin{ex}
  Let $R=K[ [t^9, t^{10}, t^{11}, t^{23}]]$, $Q=(t^9)R$ and $\frakm=(t^9, t^{10}, t^{11}, t^{23})R$. Then $R$ is Gorenstein and we have $\rms_Q(\frakm)=\rmr_Q(\frakm)=4$. But
  \[
  t^{34}=t^{11} t^{23} \in \frakm( (\frakm^6:_R\frakm^4)\cap \frakm)\setminus \frakm^3.
  \]
  Thus, by \citet[Corollary 2.3 and Remark 2.7]{DANNA2009}, $\gr_\frakm(R)$ is not even Buchsbaum.
\end{ex}

If the 1-dimensional local ring $R$ is not associated to any numerical semigroup, then the theorem might still fail, even when $R$ has embedding dimension 3. The prototype of the following example is due to Lance Bryant.

\begin{ex}
    The computer algebra system {\sc Singular} \cite{GPS} suggests that the ideal $I=(a^3+c^5+b^6,a^2b+ac^3+b^6)$ is a prime ideal in the polynomial ring $\QQ[a,b,c]$. Let $R=\QQ[ [a,b,c]]/IR$. The initial form ideal $I^*=(b^2c^5+ac^6, abc^5, a^2c^3, a^2b, a^3)$, hence $\QQ[a,b,c]/I^*$ is not Cohen-Macaulay. On the other hand, $Q=(b-c)R$ is a principal reduction of the maximal ideal $\frakm=(a,b,c)R$. It is not difficult to see that $\rmr_Q(\frakm)=\rms_Q(\frakm)=6$.
\end{ex}

\subsection{Buchsbaumness and $2$-Buchsbaumness}
Recall that for a one-dimensional standard graded ring $A$ with the unique homogeneous maximal ideal $\calM$, a finitely generated $A$-module $M$ is called \Index{$k$-Buchsbaum} if $\calM^k \cdot H_\calM^0(M)=0$. The 1-Buchsbaum condition is simply called \Index{Buchsbaum}, and 0-Buchsbaum modules are precisely the Cohen-Macaulay modules.

In this subsection, we will mainly investigate the Buchsbaum and $2$-Buchsbaum property of $\gr_\frakm(R)$, where $(R,\frakm)$ is a numerical semigroup ring of embedding dimension $3$. Denote the homogeneous maximal ideal of $\gr_\frakn(S)\isom K[x_1,x_2,x_3]$ by $\calM$. Since $\gr_\frakm(R)=\gr_\frakn(S)/I^*$, we will write the image of $f\in \gr_\frakn(S)$ in $\gr_\frakm(R)$ as $\overline{f}$. For the local cohomology module $H_\calM^0(\gr_\frakm(R))$, we can also replace $\calM$ by the homogeneous maximal ideal of $\gr_\frakm(R)$. Let $r$ be the reduction number of $\frakm$. Then it is not very difficult to show that $H_\calM^0(\gr_\frakm(R))=(0:_{\gr_\frakm(R)}\calM^r)$ (cf.~\citet[Lemma 2.2]{DANNA2009}). Therefore, $\gr_\frakm(R)$ is Buchsbaum if and only $H_\calM^0(\gr_\frakm(R))=(0:_{\gr_\frakm(R)}\calM)$.

Sapko investigated the tangent cone of numerical semigroup rings and made the following conjectures regarding the Buchsbaumness.
\begin{conj}
    [{\cite{MR1855122}}]
    Let $(R,\frakm)$ be a numerical semigroup ring of embedding dimension $3$.
    \begin{enumerate}[a]
        \item If $\gr_\frakm(R)$ is Buchsbaum, then the initial form ideal $I^*$ of $I$ is generated by 4 elements, and for some integer $k\ge 1$,
    \[
    (0:_{\gr_m(R)}\calM)=(\overline{x_3^k}) \gr_\frakm(R).
    \]
\item $\gr_\frakm(R)$ is Buchsbaum if and only if  $\length(H_\calM^0(\gr_\frakm(R)))\le 1$.
    \end{enumerate}
    \end{conj}

The main theme of this subsection is to confirm the above conjectures, and prove similar results when the tangent cone is 2-Buchsbaum.

\begin{lem}
    \Label{principal}
    Let $(R,\frakm)$ be a numerical semigroup ring of embedding dimension $3$. If $\gr_\frakm(R)$ is not Cohen-Macaulay and $\calM$ is the homogeneous maximal ideal of $\gr_\frakm(R)$, then the 0-th local cohomology module $H_\calM^0(\gr_\frakm(R))$ is principal and is generated by $\overline{x_3^\gamma}$ for suitable $\gamma\in \NN$.
\end{lem}

\begin{proof}
    Recall that given a degree-compatible local monomial order like the $>_{ds}$, the initial form ideal $I^*$ is generated by the initial forms of a binomial standard basis of $I$.
    Since $n_1<n_2<n_3$, $I^*$ is generated by forms of the following 4 types, with all visible exponents strictly positive:
    \begin{enumerate}[a]
        \item $x_3^{\alpha_3}$,
        \item $x_2^{\gamma_2}$ or a balanced binomial $x_2^{\gamma_2}-x_1^{\gamma_{21}}x_3^{\gamma_{23}}$,
        \item $x_1^{a}x_3^{c}$,
        \item $x_2^{b}x_3^{c}$.
    \end{enumerate}
    For any minimal generating set, there is exactly one generator of type (a). The same  is true for generators of type (b). To see this, it suffices to notice that if $x_2^{\gamma_2}-x_1^{\gamma_{21}}x_3^{\gamma_{23}}$ is balanced, then $x_2^{\gamma_2}$ is its leading monomial. On the other hand, there might be more than one generators of type (c) or (d).

    It follows from Lemma \ref{CM-cond-1} that $I^*$ is Cohen-Macaulay if and only if generators of type (c) do not exist. If $\gr_\frakm(R)$ is not Cohen-Macaulay, then $ H_\calM^0(\gr_\frakm(R))\ne 0$. We claim that this local cohomology module is generated by $\overline{x_3^\gamma}$ where
    \[
    \gamma = \min \Set{c \mid x_1^ax_3^c \text{ is a generator of $I^*$ of type (c) for some nonzero } a,c\in \NN}.
    \]
    Since $\sqrt{I^*}=(x_2,x_3)$, $\overline{x_3^\gamma} \in H_\calM^0(\gr_\frakm(R))$. On the other hand, $I^*+(x_3^\gamma)$ is (not necessarily minimally) generated by $x_3^\gamma$ together with the remaining generators of $I^*$ of type (b) or (d). This last ideal is Cohen-Macaulay by  Lemma \ref{CM-cond-1}. Hence the local cohomology module $H_\calM^0(\gr_\frakm(R))$ is generated by $\overline{x_3^\gamma}$.
\end{proof}

For a one-dimensional Cohen-Macaulay local ring $(R,\frakm)$, when $\length(H_\calM^0(\gr_\frakm(R)))\le 1$, one sees immediately that the associated graded ring $\gr_\frakm(R)$ is Buchsbaum. When $R$ is a numerical semigroup ring of embedding dimension $3$, the previous lemma implies that the converse is also true.

\begin{thm}
    \label{length1}
 Let $(R,\frakm)$ be a numerical semigroup ring of embedding dimension $3$.  Then $\gr_\frakm(R)$ is Buchsbaum if and only if $\length(H_\calM^0(\gr_\frakm(R)))\le 1$.
 \end{thm}

Next, we study the 2-Buchsbaumness of the tangent cone. When $\length(H_\calM^0(\gr_\frakm(R)))\le 2$, the associated graded ring $\gr_\frakm(R)$ is clearly $2$-Buchsbaum. We found that the converse is also true for a numerical semigroup ring of embedding dimension $3$.

\begin{thm}
    \Label{length2}
    Let $(R,\frakm)$ be a numerical semigroup ring of embedding dimension $3$. Then $\gr_\frakm(R)$ is 2-Buchsbaum if and only if
    $\length(H_\calM^0(\gr_\frakm(R)))\le 2.$
\end{thm}

\begin{proof}
    It suffices to assume that $\gr_\frakm(R)$ is 2-Buchsbaum, not Cohen-Macaulay, and investigate the length of the local cohomology module. Lemma \ref{principal} guarantees a monomial minimal generator $x_1^a x_3^c$ in $I^*$.  Since $x_1^2 x_3^c, x_3^{2+c} \in I^*$, we have $1\le a \le 2$ and $\alpha_3 -2 \le c \le \alpha_3-1$.

    We claim that there is exactly one  minimal monomial generator of $I^*$ having the form $x_1^a x_3^c$. It is easy to see that this could fail only when both $x_1x_3^{\alpha_3-1}$ and $x_1^2 x_3^{\alpha_{3}-2}$ are minimal generators of $I^*$. Since they are minimal, there exist $\beta_1,\beta_2\in \NN$ such that both $x_2^{\beta_1}-x_1 x_3^{\alpha_3-1}$ and $x_2^{\beta_2}-x_1^2 x_3^{\alpha_3-2}$ are weakly balanced binomials in $I$. Because $n_3> n_2 > n_1$, we must have $\beta_1> \beta_2$ and $x_2^{\beta_1-\beta_2}x_1=x_3$. Hence $G$ is two-generated, contradicting our assumption of $d=3$.

    Meanwhile, we notice that $x_2^2 x_3^c \in I^*$. Hence either $\alpha_2=2$ or this monomial is divisible by the leading monomial $x_2^b x_3^{c'}$ of a minimal generator of $I^*$ with $1\le b \le 2$ and $1\le c'\le c$. If $\alpha_2=2$, then Corollary \ref{d3cm} implies that $I^*$ is Cohen-Macaulay. Hence $\alpha_2> 2$ and, by an argument similar to that in the previous paragraph, there is exactly one minimal generator in $I^*$ having the form $x_2^bx_3^{c'}$ with $1\le b \le 2$ and $1\le c'\le c$.

    Now we are ready to show that $\length(H_\calM^0(\gr_\frakm(R)))\le 2$.
    \begin{enumerate}[a]
        \item Suppose that $x_1^a x_3^c=x_1 x_3^{\alpha_3-2}$. Notice that $x_2^2x_3^{\alpha_3-2},x_2x_3^{\alpha_3-1}\in I^*$. Each of them has to be divisible by some monomial minimal generator of $I^*$ of the form $x_2^b x_3^{c'}$ with $1\le b \le 2$, $1\le c'\le c$. But there is at most one such generator. Hence this generator must divide the $\gcd(x_2^2 x_3^{\alpha_3-2},x_2x_3^{\alpha_3-1})=x_2x_3^{\alpha_3-2}$. In particular, $x_2x_3^{\alpha_3-2}\in I^*$.  Consequently the vector space $H_\calM^0(\gr_\frakm(R))=(\overline{x_3^{\alpha_3-2}})\gr_\frakm(R)$ is generated by $\Set{x_3^{\alpha_3-2},x_3^{\alpha_3-1}}$.

        \item The case $x_1^a x_3^c=x_1^2 x_3^{\alpha_3-2}$ can never happen. Notice that the image of $x_3^{\alpha_3-2}$ generates the local cohomology module. Hence $x_1x_3 \cdot x_3^{\alpha_3-2}\in I^*$. We know that there cannot exist two distinct minimal generators of the form $x_1^\alpha x_3^\gamma$ in $I^*$. Since $x_1^2 x_3^{\alpha_3-2}$ is assumed to be a minimal generator, it has to divide $x_1 x_3^{\alpha_3-1}$, which is impossible.

        \item Assume that $x_1^a x_3^c=x_1 x_3^{\alpha_3-1}$. Notice that $x_3^{\alpha_3}\in I^*$. Hence the local cohomology module is generated as a vector space by $\Set{x_3^{\alpha_3-1}}$ or $\Set{x_3^{\alpha_3-1},x_2x_3^{\alpha_3-1}}$.

        \item Assume that $x_1^a x_3^c=x_1^2 x_3^{\alpha_3-1}$. We have $x_1x_2 x_3^{\alpha_3-1}\in I^*$ by the 2-Buchsbaumness. Since $x_1x_2x_3^{\alpha_3-1}$ is not a minimal generator, either $x_1 x_3^{\alpha_3-1}\in I^*$ or $x_2 x_3^{\alpha_3-1}\in I^*$. Because $x_1^2 x_3^{\alpha_3-1}$ is a minimal generator, the first option cannot happen. Hence $x_2 x_3^{\alpha_3-1}\in I^*$ and the local cohomology module is generated  as a vector space by $\Set{x_3^{\alpha_3-1},x_1x_3^{\alpha_3-1}}$.
    \end{enumerate}
\end{proof}

\begin{lem}
    \Label{BasicLemma}
     Suppose $(R,\frakm)$ is a Gorenstein numerical semigroup ring with embedding dimension $d=3$ and $\gr_\frakm(R)$ is $2$-Buchsbaum, then $\gr_\frakm(R)$ is indeed Cohen-Macaulay.
\end{lem}

\begin{proof}
   Suppose to the contrary that $\gr_\frakm(R)$ is not Cohen-Macaulay. Then $\alpha_2\ge 3$ and it follows from the previous proof that  $I^*$ has exactly one minimal generator of the form $x_1^{\gamma_1}x_3^{\gamma_3}$ with $\gamma_1,\gamma_3>0$. Furthermore, $\gamma_1=1$ or $2$, and $\alpha_3=\gamma_3+1$ or $\gamma_3+2$. Since $x_2^2 x_3^{\alpha_3-1}\in I^*$, there exists a generator $f=x_2^\beta x_3^\gamma -x_1^\alpha$ belonging to the binomial reduced standard basis of $I$ with $\beta\le 2$ and $\gamma \le \alpha_3-1$. Since $\gamma< \alpha_3$ and $\alpha_2\ge 3$, we have $\beta>0$ and this $f$ is not a new generator generated from the standard basis algorithm. Instead, it has to be one of the minimal binomial generators of $I$.

    By Theorem \ref{herzog-thm}, when $R$ is Gorenstein, the defining ideal, after a permutation $(i,j,k)$ of $(1,2,3)$, is
    \[
    I=(x_i^{\alpha_i}-x_j^{\alpha_j}, x_k^{\alpha_k}-x_i^{\alpha_{ki}}x_j^{\alpha_{ki}}).
    \]
    By symmetry, we can always assume that $i<j$.  Now one can  characterize when the associated graded ring is 2-Buchsbaum in terms of these $\alpha$'s.
    By our discussion for $x_2^\beta x_3^\gamma$, we only need the check the case where $(i,j,k)=(2,3,1)$, whence
    \[
    I=(f_1:=x_3^{\alpha_3}-x_2^{\alpha_2}, f_2:=x_1^{\alpha_{1}}-x_2^{\alpha_{12}}x_3^{\alpha_{13}}).
    \]
    It is evident that $f_1^*=x_3^{\alpha_3}$ and we can assume that $0\le \alpha_{13}< \alpha_3$, hence $\alpha_{12}>0$. Now $f_2^{*}=-x_2^{\alpha_{12}}x_3^{\alpha_{13}}$ and it is non-comparable with $f_1^*$. Applying the standard basis algorithm, we get $f_3:=\spoly(f_1,f_2)=-x_2^{\alpha_2 + \alpha_{12}}+x_1^{\alpha_1}x_3^{\alpha_3-\alpha_{13}}$, which must belong to the reduced standard basis of $I$.  Notice that  $I^*$ is perfect if and only if $\alpha_2 +\alpha_{12}\le \alpha_1 + \alpha_3 -\alpha_{13}$.  Since we have assumed that $I^*$ is not perfect, $f_3^*=x_1^{\alpha_{1}}x_3^{\alpha_{3}-\alpha_{13}}$. Now by our discussion above, $\alpha_{1}\le 2$. But if $G$ is minimally generated by 3 elements, then $\alpha_1>2$,  and this is a contradiction.  Thus, $\gr_\frakm(R)$ is Cohen-Macaulay.
\end{proof}

\begin{prop}
    \Label{2-Buch-4}
    Suppose $(R,\frakm)$ is a numerical semigroup ring of embedding dimension $3$ and $\gr_\frakm(R)$ is 2-Buchsbaum, then the initial form ideal $I^*$ is $4$-generated.
\end{prop}

\begin{proof}
    We may assume that $\gr_\frakm(R)$ is not Cohen-Macaulay. Hence by the proofs of Theorem \ref{length2} and Lemma \ref{BasicLemma}, we have $\alpha_2 \ge 3$ and $R$ is not Gorenstein. Now it suffices to show that $I$ has exactly one more standard basis element in addition to its $3$ minimal binomial generators.

    \begin{enumerate}[a]
        \item Suppose that $x_1x_3^{\alpha_3-2}$ is a minimal generator for $I^*$, then
            \[
            I=(f_1:=x_1^{\alpha_1}-x_2x_3^2, f_2:=x_2^{\alpha_2}-x_1x_3^{\alpha_3-2}, f_3:=x_3^{\alpha_3}-x_1^{\alpha_1-1}x_2^{\alpha_2-1})
            \]
            by case (a) of the proof for \ref{length2} together with Theorem \ref{herzog-thm}. We observe that $\spoly(f_1,f_3)$ and $\spoly(f_1,f_2)$ do not contribute to the standard basis.  Since $\gr_\frakm(R)$ is not Cohen-Macaulay, $\alpha_2 > 1+ (\alpha_3-2)$.

            Because all exponents are strictly positive, $\alpha_3-2\ge 1$. If $\alpha_3=3$, then $\overline{x_3}$ generates the local cohomology module and $x_2^2 x_3 \in I^*$. Thus there is a generator $f=x_2^\beta x_3 -x_1^\gamma \in I$ in the standard basis  with $\beta=1$ or $2$. Observe that $f_4:=\spoly(f_1,f_2)=-x_2^{\alpha_2+1}x_3+x_1^{\alpha_1+1}$. Since $n_2> n_1$ and $\alpha_2\ge 3>\beta$,  this will imply that $\gamma < \alpha_1$, which contradicts the choice of $\alpha_1$.

            Hence $\alpha_3  \ge 4$ and $f_4:=\spoly(f_2,f_1)=x_2^{\alpha_2+1}-x_1^{\alpha_1+1}x_3^{\alpha_3-4}$. By the 2-Buchsbaumness, $x_3^{\alpha_3-4}$ is not the generator for the local cohomology module and $x_2^{\alpha_2+1}$ has to be the leading monomial. The standard basis algorithm will stop at this step.

        \item Suppose that $x_1x_3^{\alpha_3-1}$ is a minimal generator for $I^*$. Then
            \[
            I=(f_1:=x_1^{\alpha_1}-x_2^{\alpha_{12}}x_3,f_2:=x_2^{\alpha_2}-x_1 x_3^{\alpha_{3}-1}, f_3:=x_3^{\alpha_3}-x_1^{\alpha_1-1}x_2^{\alpha_{2}-\alpha_{12}}),
            \]
            with $\alpha_{12}=1$ or $2$. The standard basis algorithm generates $f_4:=\spoly(f_2,f_1)=x_2^{\alpha_{2}+\alpha_{12}}-x_1^{\alpha_1+1} x_3^{\alpha_3-2}$. If the tangent cone is 2-Buchsbaum, then $\alpha_2+\alpha_{12}\le \alpha_1+\alpha_3-1$. And then the algorithm stops at this step.

        \item If $x_1^2 x_3^{\alpha_3-1}$ is a minimal generate for $I^*$, then by the proof for Theorem \ref{length2}, $\alpha_{12}=1$ and the defining ideal is
            \[
            I=(f_1:=x_1^{\alpha_1}-x_2 x_3, f_2:=x_2^{\alpha_2}-x_1^2 x_3^{\alpha_3-1},f_3:=x_3^{\alpha_3}-x_1^{\alpha_1-2}x_2^{\alpha_2-1}).
            \]
            Similar to the previous case, the standard basis algorithm will only contribute an additional basis element $f_4:=\spoly(f_2,f_1)=x_2^{\alpha_2+1}-x_1^{\alpha_1+2}x_3^{\alpha_3-2}$.
    \end{enumerate}

\end{proof}

\begin{ex}
\label{Buchrs-Converse}
    Let $K$ be a field, $R=K[ [t^5,t^6,t^{14}]]$ and $\frakm=(t^5,t^6,t^{14})R$. For this numerical semigroup ring,
the defining ideal
$
I=(x_1^4-x_2x_3,x_2^4-x_1^2x_3,x_3^2-x_1^2x_2^3)\subseteq K[ [x_1,x_2,x_3]]$
and  the initial form ideal $I^*=(x_2x_3,x_1^2x_3,x_3^2,x_2^5)\subseteq K[x_1,x_2,x_3]$, which is 4-generated. Let $\calM$ be the homogeneous maximal ideal of $\gr_\frakm(R)$.  Then the local cohomology module $H_\calM^0(\gr_\frakm(R))$ is generated by the element $\overline{x_3}$ in $\gr_\frakm(R)$ as an Artinian $R$-module. It is also generated by the elements $\overline{x_3}$ and $\overline{x_1x_3}$ as a $K$-vector space. Therefore, $\gr_\frakm(R)$ is $2$-Buchsbaum, but not Buchsbaum.
\end{ex}

\begin{prop}
    \Label{Buchrs}
    Suppose $(R,\frakm)$ is a numerical semigroup ring of embedding dimension $3$. If $\gr_\frakm(R)$ is Buchsbaum, but not Cohen-Macaulay, then for the principal reduction ideal $Q=(t^{n_1})R$ of $\frakm$ and the $\alpha$-invariant $\alpha_2$, $\rmr_Q(\frakm)=\alpha_2=\rms_Q(\frakm)+1$.
\end{prop}

\begin{proof}
    Let $r=\rmr_Q(R)$ be the reduction number and $s=\rms_Q(R)$ the index of nilpotency.
    Since $\gr_\frakm(R)$ is Buchsbaum,  by Lemma \ref{principal}, $H_\calM^0(\gr_\frakm(R))$ is generated by $\overline{x_3^{\alpha_3-1}}$, $\alpha_{12}=\alpha_{21}=1$ and $\alpha_{23}=\alpha_3-1$. Now by Theorem \ref{herzog-thm}, $\alpha_{13}=1$,  $\alpha_{31}=\alpha_1-1$ and $\alpha_{32}=\alpha_2-1$. Thus the defining ideal $I$ has the following standard basis:
    \begin{enumerate}[a]
        \item $f_1:=x_1^{\alpha_1}-x_2 x_3$ with $\alpha_1\ge 3$,
        \item $f_2:=x_2^{\alpha_2}-x_1 x_3^{\alpha_3-1}$ with $\alpha_2 \ge \alpha_3+1$,
        \item $f_3:=x_3^{\alpha_3}-x_1^{\alpha_1-1} x_2^{\alpha_2-1}$ with $\alpha_3 \le \alpha_1 + \alpha_2 -3$,
        \item $f_4:=x_2^{\alpha_2+1}-x_1^{\alpha_1+1}x_3^{\alpha_3-2}$ with $\alpha_2\le \alpha_1+\alpha_3-2$.
    \end{enumerate}
    The inequality in case (d) follows from the fact that $I$ has only $4$ standard basis, hence the standard basis algorithm has to stop after generating $f_4$.
    Since $\alpha_2\ge \alpha_3+1$, by using the binomials $f_1,f_3$ and $f_4$, it is straightforward to show that $(\overline{x_2},\overline{x_3})^{\alpha_2+1}\subseteq \overline{x_1} \frakm^{\alpha_2}$. Hence $\frakm^{\alpha_2+1}=\overline{x_1}\frakm^{\alpha_2}$ and  $r\le \alpha_2$. On the other hand, it follows from the definition of $\alpha_2$ that $\overline{x_2^{\alpha_2-1}}\not\in (\overline{x_1})$. Hence $s\ge \alpha_2-1$.  Since $\gr_\frakm(R)$ is not Cohen-Macaulay, Theorem \ref{CM-cond-2} implies that $r>s$. These three inequalities lead to $r=\alpha_2=s+1$.
\end{proof}

\begin{ex}
    \Label{Buch-Ex}
    Let $K$ be a field, $R=K[ [t^5,t^6,t^{19}]]$, $Q=(t^5)R$ and $\frakm=(t^5,t^6,t^{19})R$. The defining ideal is
    \[
    I=(x_1^5-x_2 x_3, x_2^4 - x_1 x_3, x_3^2-x_1^4 x_2^3) \subseteq K[[x_1,x_2,x_3]].
    \]
    The initial form ideal is $I^*=(x_2^5,x_3^2,x_2 x_3, x_1 x_3) \subseteq K[x_1,x_2,x_3]$. For the homogeneous maximal ideal $\calM$ of $\gr_\frakm(R)$, the local cohomology module $H_\calM^0(\gr_\frakm(R))$ is generated by the element $\overline{x_3}$ in $\gr_\frakm(R)$ as a $K$-vector space. Hence $\gr_\frakm(R)$ is Buchsbaum, but not Cohen-Macaulay. We have $\rmr_Q(\frakm)=\alpha_2=4$ and $\rms_Q(\frakm)=3$.
\end{ex}

\begin{ex}
  The converse of Proposition \ref{Buchrs} is not true. In Example \ref{Buchrs-Converse}, $Q=(t^5)R$ is a principal reduction of the maximal ideal $\frakm=(t^5,t^6,t^{14})R$, satisfying $\rmr_Q(\frakm)=\alpha_2=4$ and $\rms_Q(\frakm)=3$. However, the tangent cone $\gr_\frakm(R)$ is $2$-Buchsbaum, but not Buchsbaum.
\end{ex}

At the end of this subsection, we give another characterization of Buchsbaumness of the tangent cone $\gr_\frakm(R)$, with a different flavor from that of Theorem \ref{length1}.

Let $G=\ideal{n_1,\dots,n_d}$ be a general numerical semigroup with multiplicity $e=n_1$. If the associated semigroup ring is $(R,\frakm)$ with $\Quot(R)$ being  the total quotient ring and $r$ being the reduction number of $\frakm$,
 then the numerical semigroup $G'$ of the blowup ring $R':=\bigcup_{n\ge 1}(\frakm^n :_{\Quot(R)} \frakm^n) = (\frakm^r:_{\Quot(R)} \frakm^r)$ is $\ideal{n_1,n_2-n_1,n_3-n_1,\dots,n_d-n_1}$ (cf.~\citet[Section 3.3]{MR2265800}).

Let $\Ap(G,e)=\Set{w_0,\dots,w_{e-1}}$ be the {\Apery} set of $G$ with respect to $e$, where $w_i$ is the smallest element in $G$ congruent to $i$ modulo $e$. Similarly, let $\Ap(G',e)=\Set{w_0',\dots,w_{e-1}'}$. Furthermore, let $M=G\setminus \Set{0}$ be the maximal ideal of the semigroup $G$. In \citet{MR2256395} and \citet{DANNA2009} the following invariants for $G$ were defined. For each $i=0,1,\dots,e-1$, let $a_i=(w_i-w_i')/e$, $b_i=\max\Set{n\mid w_i\in nM}$, $c_i=\min\Set{n\mid w_i' \in nM - ne}$ and $d_i=\min\Set{n\mid w_i' \in nM -nM}$. All these invariants are non-negative integers.

\begin{thm}
    [{\cite[Theorem 2.6]{MR2256395}}]
    \label{CMChar}
    The tangent cone $\gr_\frakm(R)$ is Cohen-Macaulay if and only if $a_i=b_i$ for each $i=0,1,\dots,e-1$.
\end{thm}

\begin{prop}
    [{\cite[Proposition 3.5]{DANNA2009}}]
    \label{Invariants}
 We always have   $b_i\le a_i \le c_i \le d_i \le r$, where $r$ is the reduction number of the maximal ideal. Moreover, $b_i < a_i$ if and only if $a_i<c_i$.
\end{prop}

\begin{thm}
    [{\cite[Theorem 3.8]{DANNA2009}}]
    \label{Buch-Suff}
    Suppose $d_i=a_i+1$ for every $i$ such that $a_i>b_i$. Then $\gr_\frakm(R)$ is Buchsbaum.
\end{thm}

We want to show that the condition in Theorem \ref{Buch-Suff} is also necessary when the embedding dimension $d=3$.

\begin{thm}
    \label{Buch-Another-Char}
    Let $(R,\frakm)$ be a numerical semigroup ring of embedding dimension $3$. Then the associated graded ring $\gr_\frakm(R)$ is Buchsbaum if and only if $d_i=a_i+1$ for every $i$ such that $a_i>b_i$.
\end{thm}

\begin{proof}
    By Theorems \ref{CMChar} and \ref{Buch-Suff}, we may assume that $\gr_\frakm(R)$ is Buchsbaum, but not Cohen-Macaulay, and prove that $d_i=a_i+1$ for every $i$ such that $a_i>b_i$.

    Let $M$ be the maximal ideal of the semigroup $G$, $e=n_1$ the multiplicity of $R$ and $r$ the reduction number of $\frakm$. From the discussion in \citet[Section 3]{DANNA2009} we know that the blowup semigroup $G'=rM-re$. Furthermore, Remark 3.3 of \citet{DANNA2009} says that $a_i>b_i$ if and only if there exists $s'\equiv {i} \pmod{e}$ in $G'$ such that $s'+(h+1)e \in hM \setminus (h+1) M$ for some non-negative integer $h$.  Since $s'\in G'$, $s'+rM\subseteq rM$. Hence if $s'+(h+1)e \in hM \setminus (h+1) M$, then $s'+(h+1)e + r M \subseteq (h+1+r)M$, thus $\overline{t^{s'+(h+1)e}}:=t^{s'+(h+1)e}+ \frakm^{h+1} \in H_\calM^0(\gr_\frakm(R))$.

    Recall that $\alpha_3=\min\Set{\alpha \in \NN \mid \alpha  n_3\in \ideal{n_1,n_2}, \alpha \ne 0}$. Since $G$ is three-generated and $\gr_\frakm(R)$ is Buchsbaum, Lemma \ref{principal} shows that  $H_\calM^0(\gr_\frakm(R))$ is the $R/\frakm$-vector space generated by $\overline{x_3^{\alpha_3-1}}$. Whence $s'+(h+1)e = (\alpha_3-1)n_3$.  For this reason, there exists a unique $s'\in G'$ such that $s'+(h+1)e \in h M \setminus (h+1)M$ for some $h\in \NN$, and if $a_i>b_i$, then $s'\equiv i \pmod{e}$.  Fix this $i$. By virtue of Proposition \ref{Invariants}, now it suffices to show that $a_i+1=r$.

    Let $\alpha_2$ be the invariant that is defined similarly to $\alpha_3$. Proposition \ref{Buchrs} shows that $r=\alpha_2$.   By the definition of $\alpha_3$, $t^{(\alpha_3-1)n_3}\not \in (t^{n_1})R$, hence the {\Apery} element $w_i$ equals $(\alpha_3-1)n_3$. Notice that $a_i<r$. Hence, in order to show that $a_i=r-1$, it suffices to show that $w_i-(r-1)e\in G'=rM -re$, or equivalently, $(\alpha_3-1)n_3 + e \in \alpha_2 M$. But this follows from the binomial $f_2=x_2^{\alpha_2}-x_1 x_3^{\alpha_3-1}$ in the proof of Proposition \ref{Buchrs}.
\end{proof}

\begin{ex}
   Assume the notation in Example \ref{Buch-Ex}. We have already known that $\rmr(\frakm)=4$ and $\gr_\frakm(R)$ is Buchsbaum. For the semigroup $G=\ideal{5,6,19}$, the {\Apery} set is $\Ap(G,5)=\Set{0,6,12,18,19}$. The blowup semigroup is $G'=\Set{0,1,\rightarrow}$, hence $\Ap(G',5)=\Set{0,1,2,3,4}$. The invariants are $a=\Set{0,1,2,3,3}$, $b=\Set{0,1,2,3,1}$, $c=\Set{0,1,2,3,4}$ and $d=\Set{0,4,4,4,4}$. Notice that $a_i>b_i$ if and only if $i=4$, and $d_4=a_4+1$.
\end{ex}

\begin{rem}
    The numerical semigroup $G=\ideal{10,17,23,82}$ given by \citet[Remark 3.9]{DANNA2009} shows that Theorem \ref{Buch-Another-Char} fails if we allow the embedding dimension to be $4$.
\end{rem}

We conclude this paper by an additional remark.

\begin{rem}
The standard basis method in this paper turns out to be less fruitful if the embedding dimension $d\ge 4$. However, when $d=4$ and the tangent cone $\gr_\frakm(R)$ is Gorenstein, we are able to provide further insights with the help of linkage theory. For instance, the initial form ideal $I^\ast$ satisfies $\mu(I^*)\le 5$. This echoes a result of \citet{MR0414559}: for every Gorenstein numerical semigroup ring $(R,\frakm)$ of embedding dimension $4$, the defining ideal $I$ satisfies $\mu(I)\le 5$. Detailed discussion is available in \citet{Shen-Thesis}.
\end{rem}

\section*{Acknowledgement}
This paper is part of my Ph.D. thesis at Purdue University, which was written under the supervision of Professor Bernd Ulrich. I want to express my sincere gratitude to Professor Bernd Ulrich for the advising, encouragement and inspiration. I am also grateful to Dr.~Lance Bryant and Professor William Heinzer for the stimulating comments during the preparation of this work. I want to acknowledge the support provided by {\tt GAP} \cite{GAP4} and {\sc Singular} \cite{GPS}. In addition, I thank Dr.~Lance Bryant for bringing to my attention the research of V.~Sapko, and for his {\sc Singular} library that facilitates the computations of initial form ideals. Finally, I thank the referee for the careful reading and valuable comments and suggestions.

\end{document}